\documentclass{amsart}
\usepackage{latexsym,amsmath}
\usepackage{amsthm,amssymb,mathrsfs}

\newcommand{\RE}{\mathrm{Re}\,}
\newcommand{\im}{\mathrm{Im}\,}
\newcommand{\Aut}{\mathrm{Aut}}
\newcommand{\Isot}{\mathrm{Isot}}

\newcommand{\RR}{\mathbb{R}} 
\newcommand{\CC}{\mathbb{C}} 

\newcommand{\BB}{\mathbb{B}}
\newcommand{\PP}{\mathbb{P}}

\newcommand{\beq}{\begin{eqnarray*}}
\newcommand{\eeq}{\end{eqnarray*}}
\newcommand{\beg}{\begin{equation}}
\newcommand{\eeg}{\end{equation}}

\newtheorem{thm}{Theorem}[section]
\newtheorem{cor}[thm]{Corollary}
\newtheorem{lem}[thm]{Lemma}
\newtheorem{prop}[thm]{Proposition}
\theoremstyle{definition}
\newtheorem{definition}[thm]{Definition}
\theoremstyle{remark}
\newtheorem{remark}[thm]{Remark}
\newtheorem{example}[thm]{Example}

\numberwithin{equation}{section}

\makeatletter

\newcommand{\Rmnum}[1]{\expandafter\@slowromancap\romannumeral #1@}
\makeatother

\title{proper holomorphic polynomial maps between bounded symmetric domains of classical type}
\date\today

\author{Aeryeong Seo}

\address{School of Mathematics,
Korea Institute for Advanced Study (KIAS)
85 Hoegiro, Dongdaemun-gu, Seoul 130-722, Korea}
\email{Aileen83@kias.re.kr}


\begin{document}
\maketitle

\markboth{Aeryeong Seo}{proper holomorphic polynomial maps between BSDs of classical type }

\begin{abstract}
We prove that two proper holomorphic polynomial maps
between bounded symmetric domains of classical type
which preserve the origin are equivalent if and only if they are isotropically equivalent.
Using this property we show that each member of a one-parameter family of maps from \cite{Seo}
is inequivalent.
\end{abstract}

\section{Introduction}
Let $\Omega_1$, $\Omega_2$ be domains in $\CC^n$ and $\CC^N$ and $f,\, g : \Omega_1 \rightarrow \Omega_2$ be holomorphic maps. We say that $f$ is \textit{proper} if $f^{-1}(K)$ is compact for every compact subset $K\subset \Omega_2$.
We say that $f$ and $g$ are \textit{equivalent} if and only if $f = A\, \circ\, g \,\circ B$ for some $B\in \Aut(\Omega_1)$ and $A\in \Aut(\Omega_2)$. For a domain $\Omega$, denote the group of automorphisms fixing $p\in \Omega$ by $\text{Isot}_p(\Omega)$.
Suppose that for fixed $p\in \Omega_1$, $f(p)=g(p)$.
Then we say that $f$ and $g$ are \textit{isotropically equivalent at p} if there are $U\in \text{Isot}_p(\Omega_1)$
and $V\in \text{Isot}_{g(p)}(\Omega_2)$ such that $f = V\circ g \circ U$.
The notion of isotropic equivalence coincides with that of unitary equivalence of \cite{D'Angelo_Michigan}
defined when $\Omega_1$ and $\Omega_2$ are balls.
The following domains are called bounded symmetric domains of classical type:
\begin{enumerate}
\item $\Omega_{r,s}^\Rmnum{1} = \{ Z\in M_{r,s}^\CC : I_r -  ZZ^* >0 \}$,  where $s \geq r = \text{rank}(\Omega^\Rmnum{1}_{r,s})$. \label{1}
\vskip 0.1cm
\item $\Omega_{n}^{\Rmnum{2}} = \{ Z\in  M_{n,n}^\CC : I_n - ZZ^*>0, \,\, Z^t = -Z \}$,\,
$\text{rank}(\Omega_{n}^{\Rmnum{2}})= \big[\frac{n}{2}\big]$. \label{2}
\vskip 0.1cm
\item $\Omega_{n}^{\Rmnum{3}} = \{ Z\in M_{n,n}^\CC : I_n - ZZ^*  >0,\,\, Z^t = Z\}$,\,
$\text{rank}(\Omega_{n}^{\Rmnum{3}})= n$. \label{3}
\vskip 0.1cm
\item $ \Omega_n^{\Rmnum{4}} = \big\{ Z=(z_1,\dots, z_n)\in \CC^n : ZZ^* <1 ,\, 0< 1-2 ZZ^*+ \big| ZZ^t \big|^2  \big\}$,\, $\text{rank}(\Omega^{\Rmnum{4}})=2$. \label{4}
\end{enumerate}
Here we denote by $M >0$ positive definiteness of square matrix $M$, by $M_{r,s}^\CC$ the set of $r\times s$ complex matrices and by $I_{r}$ the $r\times r$ identity matrix.

The aim of this paper is to prove the following theorems as a generalization of the results in \cite{D'Angelo_Michigan} which are concerned with the proper holomorphic polynomial maps between balls.
\begin{thm} \label{main theorem}
Let $\Omega_1,\,\Omega_2$ be bounded symmetric domains of classical type and $f,\,g : \Omega_1 \rightarrow \Omega_2$ proper holomorphic polynomial maps such that $f(0)=g(0)=0$.
Then $f$ and $g$ are equivalent if and only if they are isotropically equivalent at $0$.
\end{thm}
\begin{thm} \label{non-equiv theorem}
There are uncountably many inequivalent proper holomorphic maps from $\Omega_{r,s}^\Rmnum{1}$ to $\Omega_{2r-1,2s}^\Rmnum{1}$ for $r\geq 2$, $s\geq 2$.
\end{thm}

The motivation of this paper comes from generalizing the study on proper holomorphic maps between balls to bounded symmetric domains of rank greater than or equal to two.
Proper holomorphic maps between balls have been studied for a long time
since Alexander (\cite{Alexander}) proved that every proper holomorphic self-map of the $n$-dimensional unit ball $\BB^n$ with $n\geq 2$ is a holomorphic automorphism.
For proper holomorphic maps between balls with different dimensions, much work has been done, relating the maximum degree of proper holomorphic maps to the difference of dimensions between the domain ball and the target ball.
As the first work along these lines, Webter (\cite{Webster}) proved that any proper holomorphic map from $\BB^n$ to $\BB^{n+1}$ with $n\geq 3$, $C^3$-smooth up to the boundary, is equivalent to the embedding
$$f_s : (z_1,\dots,z_n)\mapsto (z_1,\dots,z_n,0).$$
 Given a proper holomorphic map $f$ from $\BB^n$ to $\BB^N$, we consider a proper holomorphic map from $\BB^n$ to $\BB^{N+k}$ defined by $z\mapsto (f(z),0,\dots,0)$ with $k$-zeros for which we will use the same notation $f$ if there is no confusion.
When $n\geq 3$ and $ N\leq 2n-2$, Faran (\cite{Faran2}) showed that it is equivalent to $f_s$ if it is extended holomorphically over the boundary.
Furthermore he precisely classified the equivalence classes of proper holomorphic maps from
$\BB^2$ to $\BB^3$ which is $C^3$-smooth up to the boundary (\cite{Faran}).
In \cite{Forstneric}, Forstneri\v{c} proved that any proper holomorphic map from $\BB^n$ to $\BB^N$
which is $C^{N-n+1}$-smooth up to the boundary is a rational map $(p_1,\dots, p_N)/q$
where $p_j$ and $q$ are holomorphic polynomials of degree at most $N^2(N-n+1)$.
Since this work has been done, much results fit into the framework of providing sharp bounds in special situations. See \cite{D'Angelo_Lebl,D'Angelo_Lebl_Peters,Lebl_Peters}, for more details.
If $N=2n-1$, there is a proper holomorphic map which is called the Whitney map $f_w : \BB^n \rightarrow \BB^{2n-1}$ defined by
$$ f_w(z_1,\dots,z_n) = (z_1,\,\dots,z_{n-1},\, z_nz_1,\, z_nz_2,\,\dots,\, z_n^2).$$
It is inequivalent to the embedding $f_s$.
Moreover Huang and Ji (\cite{Huang_Ji}) proved that any proper rational map from $\BB^n$ to $\BB^{2n-1}$ with $n \geq 3$ is equivalent to $f_s$ or $f_w$ and any proper holomorphic embedding which is $C^2$-smooth up to the boundary is equivalent to $f_s$.
If the dimension of the target domain is larger than $2n$, there are infinitely many inequivalent proper holomorphic maps. For example,
$f_\theta : \BB^n\rightarrow \BB^{2n}$ given by
 \begin{equation} \label{D'Angelo}
 f_\theta(z) = \left(z_1,\,\dots,\, z_{n-1},\, \cos\theta z_n, \,\sin\theta z_1z_n,\,\dots,\,\sin\theta z_nz_n\right)
 \end{equation}
with $0\leq \theta\leq \frac{\pi}{2}$ are found by D'Angelo (\cite{D'Angelo_Michigan}).
In \cite{D'Angelo_Michigan}, D'Angelo showed that any two proper holomorphic polynomial map from $\BB^n$ to $\BB^N$ preserving the origin are equivalent if and only if they are isotropically equivalent at the origin 
and as a consequence \eqref{D'Angelo} are inequivalent for all $0\leq \theta\leq \frac{\pi}{2}$.
Interestingly, Hamada (\cite{Hamada}) showed that any proper rational map from $\BB^n$ to $\BB^{2n}$ with $n\geq 4$ is equivalent to $f_\theta$ for some $\theta$, $0\leq \theta \leq \pi/2$ and Huang, Ji and Xu (\cite{Huang_Ji_Xu}) showed that any proper holomorphic map from $\BB^n$ to $\BB^N$ with $4\leq n \leq N\leq 3n-4$ which is $C^3$-smooth up to the boundary should be equivalent to $f_\theta$ for some $\theta$, $0\leq \theta \leq \pi/2$.
Recently Huang-Ji-Yin (\cite{Huang_Ji_Yin}) proved that any proper rational map from $\BB^n$ to $\BB^N$ with $n\geq8 $ and $3n+1 \leq N\leq 4n-7$ should be equivalent to proper rational map from $\BB^n$ to $\BB^{3n}$.

As a one generalization of the unit ball, one consider bounded symmetric domains which are Hermitian symmetric spaces of non-compact type with non-smooth boundaries. There are several rigidity theorems on proper holomorphic maps between bounded symmetric domains of rank greater than or equal two. In contrast with the case of the unit balls, the difference of the ranks between the domains is more crucial than that of the dimensions.
The first result on bounded symmetric domains along these lines is the following
which is due to Tsai.
Let $f : \Omega_1\rightarrow \Omega_2$ be a proper holomorphic map between irreducible bounded symmetric domains $\Omega_1$ and $\Omega_2$.
If $\text{rank}(\Omega_1) \geq \text{rank}(\Omega_2) $,
then $\text{rank}(\Omega_1)  = \text{rank}(\Omega_2) $ and
$f$ should be a totally geodesic isometric embedding with respect to the Bergman metrics on the domains (see \cite{Tsai}).
 If $\Omega _1 = \Omega_{r,r-1}^\Rmnum{1}$ and $\Omega_2= \Omega_{r,r}^\Rmnum{1}$,
 then $f$ is also a totally geodesic isometric embedding (see \cite{Tu}),
 although the rank of $\Omega_2$ is larger than that of $\Omega_1$.
Furthermore, Ng (\cite{Ng}) showed that for $f:\Omega_{r,s}^\Rmnum{1}\rightarrow \Omega_{r',s'}^\Rmnum{1}$, if $s\geq 2$, $s\geq r' \geq r$ and $r' \leq 2r-1$, then $f$ is equivalent to the embedding, $Z\mapsto \left(
                                                                                \begin{array}{cc}
                                                                                  Z & 0 \\
                                                                                  0 & 0 \\
                                                                                \end{array}
                                                                              \right)$.
If the difference of the ranks of the domains gets bigger, then it is expected that there are lots of inequivalent proper holomorphic maps. In \cite{Seo}, one way of finding proper holomorphic maps between bounded symmetric domains of type $\Rmnum{1}$ is suggested and several proper holomorphic maps are constructed.
For example, for $r'=2r-1$ and $s'=2s-1$, there is a generalized Whitney map defined by
\begin{equation}\label{generalized Whitney map}
 \left(
  \begin{array}{ccc}
    z_{11} & \dots & z_{1s} \\
    \vdots & \ddots & \vdots \\
     z_{r1} & \dots & z_{rs} \\
  \end{array}
\right)
\mapsto \left(
  \begin{array}{ccccccc}
    z_{11}^2  & z_{11}z_{12}&\dots& z_{11}z_{1s}& z_{12}& \dots & z_{1s}   \\
    z_{11} z_{21} & z_{21}z_{12}& \dots& z_{21}z_{1s} &z_{22}& \dots& z_{2s}  \\
    \vdots & \vdots &\ddots & \vdots&\vdots& \ddots &\vdots \\
    z_{11}z_{r1} & z_{r1}z_{12} &\dots & z_{r1}z_{1s}& z_{r2}& \dots& z_{rs} \\
    z_{21} & z_{22} &\dots & z_{2s}& 0& \dots & 0 \\
     \vdots& \vdots & \ddots& \vdots & \vdots& \ddots & \vdots\\
     z_{r1}& z_{r2}&\dots &z_{rs}&0& \dots& 0\\
  \end{array}
\right).
\end{equation}

In this paper, as a one step to observe analogous phenomenon on proper holomorphic maps between bounded symmetric domains of rank greater than or equal to two,
we generalize the result of D'Angelo in \cite{D'Angelo_Michigan} to the domains of classical type.
\subsection*{Acknowledgement}{
This research was supported by National Researcher Program of the National Research Foundation (NRF)
funded by the Ministry of Science, ICT and Future Planning(No.2014028806).
}

\section{Preliminaries}
In this section, we introduce terminology and some basic background.
A bounded domain $\Omega$ is called \textit{symmetric} if for each $p\in \Omega$,
there is a holomorphic automorphism $\mathfrak{i}_p$
such that $\mathfrak{i}_p^2$ is the identity map of $\Omega$ which has $p$ as an isolated fixed point.
All bounded symmetric domains are homogeneous domain, i.e. the automorphism group acts transitively on the domain.
In 1920's, Cartan classified all irreducible bounded symmetric domains. There are 4 classical types
and 2 exceptional types. The four classical types are given by \eqref{1},\eqref{2},\eqref{3} and \eqref{4} in the introduction.
Note that $\Omega_{m,1}^\Rmnum{1}$ is the m-dimensional unit ball and $\Omega_1^{\Rmnum{3}}$ is the unit disc.

From now on, we will use the notation $M = \left(
                                 \begin{array}{cc}
                                   A & B \\
                                   C & D
                                    \\
                                 \end{array}
                               \right)\in  GL(r+s,\CC)$ to split $M$ into 4 block matrices with $A\in M^\CC_{r,r}$, $B\in M^\CC_{r,s}$,
                               $C\in M^\CC_{s,r}$ and $D\in M^\CC_{s,s}$.
We will denote by $ASM^\CC_{n,n}$ the set of anti-symmetric complex $n\times n$ matrices
and by $SM^\CC_{n,n}$ the set of symmetric complex $n \times n$ matrices.

Let $U(r,s)$ be the subgroup of $GL(r+s,\CC)$ satisfying $$ M \left(
                                 \begin{array}{cc}
                                   -I_r & 0 \\
                                   0 & I_s \\
                                 \end{array}
                               \right) M^* = \left(
                                 \begin{array}{cc}
                                   -I_r & 0 \\
                                   0 & I_s \\
                                 \end{array}
                               \right)$$ for all $M\in U(r,s)$.
Let $SU(r,s)$ be the subset of $U(r,s)$ which consists of the matrices with determinant one.
Explicitly,
\begin{equation}\label{SU(r,s)}
\begin{split}
SU(r,s) = \bigg\{ \left(
              \begin{array}{cc}
               A & B \\
              C & D \\
           \end{array}
        \right)\in &SL(r+s,\CC) : AA^*-BB^* = I_r, \\
        &\quad AC^* = BD^*,\,\, CC^* - DD^* = -I_s \bigg\}.
\end{split}
\end{equation}
Let $O(n+2,\CC)$ be the complex orthogonal group of $(n+2)\times (n+2)$ matrices.

Since every bounded symmetric domain is Hermitian symmetric space of non-compact type,
the domain can be canonically embedded into the corresponding compact dual (Borel embedding)
and every holomorphic automorphism of the domain can be extended to the automorphism of its compact dual.
For example, let $G_{r,s}$ be the Grassmannian of $r$-planes in $r+s$ dimensional complex vector space $\mathbb{C}^{r+s}$ which is the compact dual of $\Omega_{r,s}^\Rmnum{1}$.
For $X\in M^\CC_{r,r+s}$ of rank $r$, denote $[X]$ the $r$-plane in $\mathbb{C}^{r+s}$ which is generated by row vectors of $X$.
For $M\in GL(r+s,\CC)$, $M$ acts on $ G_{r,s}$ by $[X]\in G_{r,s}  \mapsto [XM]$.
Then the Borel embedding $\xi^\Rmnum{1}$ of $\Omega_{r,s}^\Rmnum{1}$ is given by
$$\xi^\Rmnum{1}(Z)= [I_r,Z] \in G_{r,s}$$
and $M\in U(r,s)$ acts on $\Omega_{r,s}^\Rmnum{1}$ by $g_M(Z) := {\xi^\Rmnum{1}}^{-1}\left( \xi(Z)M \right)$. Explicitly, for $M= \left(
              \begin{array}{cc}
               A & B \\
              C & D \\
           \end{array}
        \right) \in U(r,s)$, $M$ acts on
$\Omega_{r,s}^{\Rmnum{1}} $
by
\begin{equation}\label{auto123}
Z\mapsto (A + ZC)^{-1}(B + ZD).
\end{equation}
Similarly, for $M=\left(
              \begin{array}{cc}
               A & B \\
              C & D \\
           \end{array}
        \right)\in U(n,n)$ acts on $\Omega_n^{\Rmnum{2}}$ and $\Omega_n^{\Rmnum{3}}$ by \eqref{auto123}.

In case of $\Omega_n^{\Rmnum{4}}$, the explicit expression of the holomorphic automorphism is little more messy.
The compact dual of $\Omega_n^{\Rmnum{4}}$ is the hyperquadric $H_n$ in $\PP^{n+1}$ given by
$H_n := \{ [z_1,\dots,z_{n+2}] \in \PP^{n+1} :  z_1^2 + \dots + z_n^2 - z_{n+1}^2 - z_{n+2}^2=0 \} $.
Then the Borel embedding $\xi^{\Rmnum{4}}$ is given by
\begin{equation}
\xi^{\Rmnum{4}}(Z) = [ -2iZ, 1+ZZ^t, i(1-ZZ^t)] \in H_n
\end{equation}
For $M=\left(
              \begin{array}{cc}
               A & B \\
              C & D \\
           \end{array}
        \right) \in O(n+2,\CC) \cap U(n,2)$, then $M$ acts on $\Omega_n^{\Rmnum{4}}$ by
$$ Z=(z_1,\dots, z_n) \mapsto \frac{1}{(-2iZB + Z'D)(i, 1)^t}(2iZA -Z'C)$$
 where $Z' =( 1+ ZZ^t, \,i-iZZ^t)$.

The automorphism groups of classical domains and their isotropy groups at the origin are given by
the following:
\begin{enumerate}
\item $\Aut\left(\Omega_{r,s}^{\Rmnum{1}}\right) = U(r,s), \,\,
 \Isot\left(\Omega_{r,s}^{\Rmnum{1}}\right) = \bigg\{  \left(
              \begin{array}{cc}
               U & 0 \\
              0 & V \\
           \end{array}
        \right) :  U\in U(r),\, V\in U(s) \bigg\}$
\item $\Aut\left(\Omega_{n}^{\Rmnum{2}}\right) = \bigg\{ M \in U(n,n) : M^t\left(
                                                                           \begin{array}{cc}
                                                                             0 & I_n \\
                                                                             I_n & 0 \\
                                                                           \end{array}
                                                                         \right)M = \left(
                                                                           \begin{array}{cc}
                                                                             0 & I_n \\
                                                                             I_n & 0 \\
                                                                           \end{array}
                                                                         \right)
 \bigg\}$, \\
        $\Isot\left(\Omega_{n}^{\Rmnum{2}}\right) = \bigg\{ \left(
              \begin{array}{cc}
               A & 0 \\
              0 & \overline{A} \\
           \end{array}
        \right) :  A \in U(n) \bigg\}$
\item $ \Aut\left(\Omega_{n}^{\Rmnum{3}}\right) = \bigg\{ M \in U(n,n) : M^t\left(
                                                                           \begin{array}{cc}
                                                                             0 & I_n \\
                                                                             -I_n & 0 \\
                                                                           \end{array}
                                                                         \right)M = \left(
                                                                           \begin{array}{cc}
                                                                             0 & I_n \\
                                                                             -I_n & 0 \\
                                                                           \end{array}
                                                                         \right)
 \bigg\}$, \\
        $\Isot\left(\Omega_{n}^{\Rmnum{3}}\right) = \bigg\{  \left(
              \begin{array}{cc}
               A & 0 \\
              0 & \overline{A} \\
           \end{array}
        \right) :  A \in U(n) \bigg\}$
\item $\Aut(\Omega_n^{\Rmnum{4}}) = O(n+2,\CC)\cap U(n,2),\quad \Isot(\Omega^{\Rmnum{4}}) = O(n)\times O(2)$
\end{enumerate}

\section{Isotropically equivalent proper holomorphic polynomial maps}
In this section, we will prove Theorem \ref{main theorem}.
Define a polynomial function $S_{r,s}^{\Rmnum{1}} : \Omega_{r,s}^{\Rmnum{1}}\rightarrow \RR $ by
 $$S_{r,s}^{\Rmnum{1}}(Z) = \det(I_r -ZZ^*)$$
 for $Z = (z_{ij}) \in M_{r,s}^\CC$, as a real polynomial in $\RE(z_{ij})$, $\im(z_{ij})$ where $1\leq i \leq r$, $1\leq j\leq s$.
$S_{r,s}^{\Rmnum{1}}$ is a polynomial of degree 2  in each $\RE(z_{ij})$, $\im(z_{ij})$.
In case of $\Omega^{\Rmnum{2}}_n$, it is known that $\det(I_r -ZZ^*) = s_n^{\Rmnum{2}}(Z)^2$ for some polynomial $s_n^{\Rmnum{2}}(Z)$ (cf.\cite{Loos}).
Define $S_n^{\Rmnum{2}}  : \Omega_n^{\Rmnum{2}} \rightarrow \RR$ and
$S_n^{\Rmnum{3}}  : \Omega_n^{\Rmnum{3}} \rightarrow \RR$ by
$$ S_n^{\Rmnum{2}} (Z) = s_n^{\Rmnum{2}}(Z) \,\,\text{ for } Z\in ASM^\CC_{n,n},$$
$$ S_n^{\Rmnum{3}} (Z) = \det(I_n -ZZ^*) \,\,\text{ for } Z\in SM^\CC_{n,n}$$
and
$$ S_n^{\Rmnum{4}}(Z) = 1-2 ZZ^* + \big| ZZ^t \big|^2 \,\,\text{ for } Z\in \CC^n$$
$S_{r,s}^{\Rmnum{1}}(Z)$, $S_n^{\Rmnum{2}} (Z)$, $S_n^{\Rmnum{3}} (Z)$ and $S_n^{\Rmnum{4}} (Z)$ are called the \textit{ generic norm } of the corresponding domains cf.\cite{Loos}.
Then $S_n^{\Rmnum{2}}(Z)$ is a polynomial of degree 2 in each $\RE{z_{ij}}, \im{z_{ij}}$ for $1 \leq i<j \leq n$
and $S_n^{\Rmnum{3}} (Z)$ is a polynomial of degree 4 in each $\RE{z_{ij}}, \im{z_{ij}}$ for $1\leq i<j \leq n$
and of degree 2 in each $\RE{z_{ii}}, \im{z_{ii}}$ for $1\leq i \leq n$.

\begin{lem} \label{coefficient}
\mbox{}
\begin{enumerate}
\item For $Z = (z_{ij})\in M_{r,s}^\CC$, the coefficient of $(\RE z_{ij})^2$ in $S_{r,s}^{\Rmnum{1}}(Z)$ is
$$-\det (I_{r-1} - Z'Z'^*)$$ where $Z'$ is the $(i,j)$ minor of $Z$. \label{I}
\item For $Z = (z_{ij})\in ASM_{n,n}^\CC$, the coefficient of $(\RE z_{ij})^4$,$1 \leq i<j \leq n$ in $S_{n}^{\Rmnum{2}}(Z)$ is
$$\det (I_{n-2} - Z''Z''^*)$$ where $Z''$ is $(n-2)\times (n-2)$ matrix obtained by removing
$i,j$-th rows and $i,j$-th columns in $Z$.\label{II}

\item For $Z = (z_{ij})\in SM_{n,n}^\CC$, the coefficient of $(\RE z_{ij})^4$,$1 \leq i<j \leq n$
 and $(\RE z_{ij})^2$ for $1\leq i\leq n$ in $S_{n}^{\Rmnum{3}}(Z)$ are
$$\det (I_{n-2} - Z''Z''^*) \quad \text{ and } \quad  -\det (I_{n-1} - Z'Z'^*)$$
respectively where $Z''$ is $(n-2)\times (n-2)$ matrix removing
$i,j$-th rows and $i,j$-th columns in $Z$ and $Z'$ is the $(i,i)$ minor of $Z$.\label{III}
\end{enumerate}
\end{lem}
\begin{proof}
\mbox{}

\textbf{\eqref{I}}. For $Z\in M_{r,s}^\CC$, denote $Z = \left(
                                     \begin{array}{c}
                                       X \\
                                       Y \\
                                     \end{array}
                                   \right)$ with $X =(X',X'') =(x_{ij})\in M_{r-1, s}^\CC$ where $X'\in M_{r-1,1}^\CC,\, X''\in M_{r-1,s-1}^\CC$ and $Y=(y_1,\dots, y_s)\in M_{1,s}^\CC$.
We only consider the coefficient of $\,(\RE y_1)^2$.
The coefficient of $(\RE y_1)^2$ in $S_{r,s}^{\Rmnum{1}}(Z)$ is $\frac{\partial^2}{\partial \overline{y}_1\partial y_1} S_{r,s}^{\Rmnum{1}}(Z)$. Since
\begin{eqnarray}
\frac{\partial}{\partial y_1}\det\left( I_{r} - ZZ^* \right) = \frac{\partial}{\partial y_1}\det \left(
                   \begin{array}{cc}
                     I_{r-1} - XX^* & -XY^* \\
                     -YX^* & 1- YY^* \\
                   \end{array}
                 \right)
\end{eqnarray}
and $YX^* = (y_1\overline{x}_{11} + \dots + y_s\overline{x}_{1s},\, \dots\,, y_1\overline{x}_{(r-1)1} + \dots + y_s\overline{x}_{(r-1)s})$, we obtain
$$
\frac{\partial}{\partial y_1}\det\left( I_{r} - ZZ^* \right) = \det  \left(
                   \begin{array}{cc}
                     I_{r-1} - XX^* & -XY^* \\
                     (-\overline{x}_{11},\, -\overline{x}_{21},\dots, -\overline{x}_{(r-1)1}) & -\overline{y}_1 \\
                   \end{array}
                 \right)
$$ and
\begin{eqnarray}
&&\frac{\partial^2}{\partial \overline{y}_1\partial y_1}\det\left( I_{r} - ZZ^* \right)\\
&=& \det \left(
                   \begin{array}{c|c}
                     I_{r-1} - XX^* &\left(
                                                        \begin{array}{c}
                                                          -x_{11} \\
                                                          \vdots \\
                                                          -x_{(r-1)1} \\
                                                        \end{array}
                                                      \right)
                      \\\hline
                     (-\overline{x}_{11},\, -\overline{x}_{21},\dots, -\overline{x}_{(r-1)1}) & -1 \\
                   \end{array}
                 \right)\label{first matrix}\\
&=& \det  \left(
                   \begin{array}{c|c}
                     I_{r-1} - X''X''^* &0
                      \\\hline
                     (-\overline{x}_{11},\, -\overline{x}_{21},\dots, -\overline{x}_{(r-1)1}) & -1 \\
                   \end{array}
                 \right)\nonumber\\
&=& -\det (I_{r-1} - X''X''^*).
\end{eqnarray}
The second equation comes from subtracting $j$-th row of \eqref{first matrix} by $x_{j1}$ times the $r$-th row of \eqref{first matrix}.

\textbf{\eqref{II}}.
We will only consider the coefficient of $(\RE{z_{12}})^4$. By Lemma \ref{coefficient}\eqref{I},
for $W = (w_{ij})\in M_{n,n}^\CC$, $\det(I_n - WW^*) = a_2 (\RE{w_{12}})^2 + a_1 (\RE{w_{12}}) + a_0$
where $a_i$ are polynomials in $\RE{w_{ij}}, \im{w_{ij}}$ for $i\neq 1, \, j\neq 2$ and $\im{w_{12}}$.
Since the coefficient of $(\RE{z_{12}})^4$ is the coefficient of $(\RE{w_{12}})^2 (\RE{w_{21}})^2$
substituted $w_{ij} = z_{ij}$ for $1\leq i <j<n$, $w_{ji} = -z_{ij}$ for $1\leq i<j\leq n$
and $w_{ii} = 0$ for $1\leq i\leq n$, the coefficient of $(\RE{z_{12}})^4$ is $\det(I_{n-2} - Z''Z''^*)$ where
$Z'' = (z_{ij})_{3\leq i\leq n,\, 3\leq j\leq n}$.

\eqref{III}. We can obtain the result by similar method in \eqref{I} and \eqref{II}.
\end{proof}

\begin{prop} \label{irreducible}
$S_{r,s}^{\Rmnum{1}}(Z)$, $S_n^{\Rmnum{2}}(Z)$, $S_n^{\Rmnum{3}}(Z)$ and $S_n^{\Rmnum{4}}(Z)$ are irreducible.
\end{prop}
\begin{proof}
\textbf{In case of $S_{r,s}^{\Rmnum{1}}$}: At first, we will prove that $S_{r,s}^{\Rmnum{1}}(Z)$ is irreducible.
Note that $r\leq s$.
We use induction with respect to $k$ on $S_{k,s-r+k}^{\Rmnum{1}}$.
For $Z\in M_{1,s-r+1}^\CC$, $1-ZZ^*$ is irreducible.
Suppose that $S_{r-1,s-1}^{\Rmnum{1}}(Z)$ is irreducible
and $S_{r,s}^{\Rmnum{1}}(Z) = A\,B$ for some polynomial $A$ and $B$. Denote $Z = \left(
                                     \begin{array}{c}
                                       X \\
                                       Y \\
                                     \end{array}
                                   \right)$ as in the proof of Lemma \ref{coefficient}.
The degree of $\RE y_1$ is 2.\\
\textbf{Step 1 :} Suppose that there is a nonzero monomial of $(\RE y_1)^2$ in $A$. Then $S_{r,s}^{\Rmnum{1}}(Z)$ should be of the form
$$ S_{r,s}^{\Rmnum{1}}(Z) = (\mu (\RE y_1)^2 + \sigma \RE{y_1} + \nu) \,B$$
where $\mu, \sigma,\nu$ and $B$ are polynomials without $\RE y_1$ variable.
Without loss of generality, we may assume that $B$ is not a constant.
Then by Lemma \ref{coefficient}, $\mu B$ is $\det(I_{r-1} - X''X''^*)$ which is irreducible by the induction hypothesis.
 Hence $B$ should be $\det(I_{r-1} - X''X''^*)$ (up to constants). This implies that it consists of the monomials of variable $X''$.
Then every variable of $X'$ should be in $A$. Consider the coefficient of $(\RE{x_{11}})^2$ which is irreducible and does not contain $\RE{x_{1j}}$, $\im{x_{1j}}$ variables. However the coefficient of  $\left(\RE{x_{11}}\right)^2$ should contain $B$ and  this induces a contradiction. Therefore there is no second order term in each $A$ and $B$.\\
\textbf{Step 2 :} Suppose that $S_{r,s}^{\Rmnum{1}}(Z) = (\mu \RE y_1 + \sigma)\,(\nu \RE y_1  + \rho)$ where $\mu, \sigma, \nu, \rho$ has no
$\RE y_1$ variable. Then $\mu \nu$ is irreducible by Lemma \ref{coefficient} and
hence $\mu$ or $\nu$ is a constant.
If $\mu$ is a constant, $\nu = \det(I_{r-1} - X''X''^*)$ up to constant.
Note that $\nu$ contains second order terms of variables in $X''$.
Since there is no $\RE y_1$ term in $\rho$, $\nu \RE y_1 + \rho$ should have second order term.
But by Step 1, $\nu \RE y_1 + \rho$ cannot have second order term.
 Hence $S_{r,s}^{\Rmnum{1}}(Z)$ is irreducible.

\textbf{In case of $S_n^{\Rmnum{2}}$, $S_n^{\Rmnum{3}}$, $S_n^{\Rmnum{4}}$:} Since for $n=1$, $S^{\Rmnum{3}}_1(z) = 1- |z|^2$ for $z\in \Delta$ which is irreducible,
use the same method (induction) as in the proof of the case $S_{r,s}^{\Rmnum{1}}$
considering factorization with respect to $\RE{z_{11}}$ variable.
In case of $S_n^{\Rmnum{4}}$, we can easily show that it is irreducible. We omit the proof.

We only need to prove that $S^{\Rmnum{2}}_n(Z)$ is irreducible. For $n=2$, let $Z = \left(
                                                                       \begin{array}{cc}
                                                                         0 & a \\
                                                                         -a & 0 \\
                                                                       \end{array}
                                                                     \right)$.
Then $S_2^{\Rmnum{2}}(Z) = 1-|a|^2$ which is irreducible.
For $n=3$, let $Z = \left(
                      \begin{array}{ccc}
                        0 & a & b \\
                        -a & 0 & c \\
                        -b & -c & 0 \\
                      \end{array}
                    \right)$.
Then $S^{\Rmnum{2}}_3(Z) = 1-|a|^2 - |b|^2 - |c|^2$ which is also irreducible.
Assume that $S^{\Rmnum{2}}_{2n}(Z)$ and $S^{\Rmnum{2}}_{2n-1}(Z)$ are irreducible.
Since the even dimensional case is similar to the odd dimensional case,
we will only consider the odd dimensional case.
Since the coefficient of $(\RE z_{ij})^4$ in $S_{2n+1}^{\Rmnum{2}}(Z)^2$
is $\det (I_{2n-1} - Z''Z''^*)$ as in Lemma \ref{coefficient} \eqref{II},
the coefficient of $(\RE z_{ij})^2$ in $S^{\Rmnum{2}}_{2n+1}(Z)$ is irreducible.
Hence similar proof of the case $S_{r,s}^{\Rmnum{1}}$ can be applied.
\end{proof}

Let $f : \Omega_1 \rightarrow \Omega_2$ be a proper holomorphic polynomial map where $\Omega_1,\, \Omega_2$ are irreducible bounded symmetric domains of classical type.
Let $S_1,\, S_2$ be the corresponding generic norms.
Since $f$ is proper, by Proposition \ref{irreducible},
$$S_2(f(Z)) = 0 \text{ whenever } S_1(Z)=0.$$
(Note that if $Z\in \partial \Omega_n^{\Rmnum{4}}$, $S^{\Rmnum{4}}_n(Z)=0$ since if $ZZ^*=1$, $S^{\Rmnum{4}}_n(Z)<0$.
Hence we do not need to consider $ZZ^*-1$ term in the definition of $\Omega_n^{\Rmnum{4}}$.)
Hence there is a real analytic map $F_f$ such that
\begin{equation}
S_2( f(Z)) = S_1(Z)\, F_f(Z).
\end{equation}
We can polarize this equation by
\begin{equation} \label{polarize}
S_2( f(Z),f(W)) = S_1(Z,W)\, F_f(Z,W).
\end{equation}
\begin{example}
\begin{enumerate}
\item Let $f: \Omega_{r,s}\rightarrow\Omega_{r',s'}$ be a proper holomorphic polynomial map.
Then \eqref{polarize} is expressed by
$$
\det(I_{r'} - f(Z)f(W)^*) = \det(I_{r} - ZW^* )\, F_f(Z,W).
$$
\item Let $f: \Omega_n^{\Rmnum{4}}\rightarrow\Omega_N^{\Rmnum{4}}$ be a proper holomorphic polynomial map. Then \eqref{polarize} is expressed by
\begin{eqnarray*}
&&1-2f(Z)f(W)^* + \left(f(Z)f(Z)^t\right)(\overline{f(W)f(W)^t})\\
&&\quad\quad\quad\quad\quad\quad= \left(1-2ZW^* + \left(ZZ^t\right)(\overline{WW^t})\right) F_f(Z,W).
\end{eqnarray*}
\end{enumerate}
\end{example}

\begin{lem}\label{composition rule}
Let $g:\Omega_1\rightarrow \Omega_2$ and $f : \Omega_2\rightarrow \Omega_3$ be
proper holomorphic polynomial maps. Then $$F_{f\circ g}(Z,W) = F_g(Z,W)\, F_f(g(Z),g(W)).$$
\end{lem}
\begin{proof}
\begin{eqnarray*}
S_1(Z, W)F_{f\circ g}(Z,W) &=& S_3( f\circ g(Z) , f\circ g(W))\\
 &=& S_2(g(Z), g(W) ) \,F_f(g(Z),g(W)) \\
 &=& S_1(Z, W)F_g(Z,W) \,F_f(g(Z),g(W))
\end{eqnarray*}
\end{proof}

\begin{lem}\label{F_U}
Let $U = \left(
           \begin{array}{cc}
             A & B \\
             C & D \\
           \end{array}
         \right)$ be an automorphism of a domain $\Omega$ with appropriate block matrices $A,B,C,D$.
         Then $F_U$ is given as the following:
\begin{enumerate}
\item If $\Omega = \Omega_{r,s}^{\Rmnum{1}}$ or $\Omega_n^{\Rmnum{3}}$,
 $$F_U(Z,W) = \frac{1}{\det(A+ZC)\,\overline{\det(A+WC)}}.$$
\item If $\Omega = \Omega_n^{\Rmnum{2}}$,
        $$F_U(Z,W) = \frac{1}{f_{U}(z,w)}\,
        \text{ where } \, f_U(Z,W)^2 = \det(A+ZC)\overline{\det(A+WC)}.$$
\item If $\Omega = \Omega_n^{\Rmnum{4}}$,
         $$F_U(Z,W) = -\frac{1}{2}\frac{1}{\{(-2iZB + Z'D)(i, 1)^t\}\{\overline{(-2iWB + W'D)(i, 1)^t}\}}$$
\end{enumerate}
\end{lem}
\begin{proof}
(1 and 2) Since by \eqref{SU(r,s)},
\begin{eqnarray*}
&&U(Z)U(W)^* = (A+ZC)^{-1}(B+ZD) \{(A+WC)^{-1}(B+WD) \}^*\\
 &=& (A+ZC)^{-1} \left( (A+ZC)(A+WC)^* + ZZ^* - I_{r}\right) \left((A+WC)^{-1} \right)^* \\
 &=& I_{r} - (A+ZC)^{-1}(I_{r}- ZZ^*) \left((A+WC)^{-1} \right)^*,
\end{eqnarray*}
\begin{equation}\label{det U}
\det(I_{r} - U(Z)U(W)^*) = \frac{\det(I_{r}- ZW^*)}{\det(A+ZC)\,\overline{\det(A+WC)}}.
\end{equation}

(3)Note that for $Z=(z_1,\dots, z_n)$, $S_n^{\Rmnum{4}}(Z) = -\frac{1}{2}Q( \tilde\xi^{\Rmnum{4}}(Z))$ where
$Q(x_1,\dots, x_{n+2}) = |x_1|^2 + \dots |x_n|^2 -|x_{n+1}|^2 -|x_{n+2}|^2 $
and $\tilde\xi^{\Rmnum{4}}(Z) = (-2iZ, 1+ ZZ^t, i(1-ZZ^t))$.
Then
\begin{eqnarray*}
S_n^{\Rmnum{4}}(g_M(Z)) &=& -\frac{1}{2}Q\left( \tilde\xi^{\Rmnum{4}}({\xi^{\Rmnum{4}}}^{-1}(\xi^{\Rmnum{4}}(Z)M)\right)\\
&=& -\frac{1}{2}Q\left( \frac{-2iZA +Z'C}{(-2iZB + Z'D)(i, 1)^t},\frac{-2iZB + Z'D}{(-2iZB + Z'D)(i, 1)^t}\right)\\
&=& -\frac{1}{2}\frac{S_n^{\Rmnum{4}}(Z)}{|(-2iZB + Z'D)(i, 1)^t|^2}
\end{eqnarray*}
\end{proof}

\begin{remark}
In case of $\Omega_n^{\Rmnum{2}}$, since for $Z\in ASM^\CC_{n,n}$, $C^tA + C^tZC$ is anti-symmetric.
This implies that
$$\det(A + ZC) = \frac{\det(C^tA + C^tZC)}{\det(C)} = \frac{(Pf(C^tA + C^tZC))^2}{\det(C)}$$
where $Pf(Y)$ is the Pfaffian polynomial of a matrix $Y$ and hence $F_U$ is a rational function.

\end{remark}
\begin{proof}[Proof of Theorem~\ref{main theorem}]

Suppose that $g \circ U = V\circ f$ for some $V = \left(
                                 \begin{array}{cc}
                                   V_1 & V_2 \\
                                   V_3 & V_4 \\
                                 \end{array}
                               \right)\in \Aut(\Omega_2)$ and
$U= \left(
                                 \begin{array}{cc}
                                   U_1 & U_2 \\
                                   U_3 & U_4 \\
                                 \end{array}
                               \right)\in \Aut(\Omega_1)$.
Then since $F_{g\circ U} = F_{V\circ f}$, by Lemma \ref{composition rule},
$$ F_g(U(Z),U(W)) F_U(Z,W) = F_V(f(Z),f(W)) F_f(Z,W).$$
By multiplying $S^1(Z,W)$ to both side, we obtain
\begin{eqnarray}\label{gU=Vf}
S^2(g\circ U(Z), g\circ U(W)) = F_V(f(Z),f(W)) S^1(f(Z),f(W)).
\end{eqnarray}
For simplicity, we only consider that the case $\Omega_1$ and $\Omega_2$ are
bounded symmetric domains of type $\Rmnum{1}$. \eqref{gU=Vf} is
\begin{equation} \label{formula}
\begin{split}
\det(I_{r'}-f(Z)f(W)^*) = \det&(I_{r'} - g\circ U(Z) (g\circ U(W))^*)\\
  &\det(V_1 + f(Z)V_3) \,\overline{\det(V_1 + f(W)V_3)}.
\end{split}
\end{equation}

Put $W=0$ in \eqref{formula}.
Then
\begin{equation}\label{formula_2}
1 =  \det(I_{r'} - g\circ U(Z) (g\circ U(0))^*)\,
  \det(V_1 + f(Z)V_3)\, \overline{\det(V_1)}.
\end{equation}
If $U_3=0$, then by $g(U_1^{-1}ZU_4) = V\circ f(Z)$,
we obtain $0=g(0)=V\circ f(0)=V_1^{-1}V_2$ and hence $V_2=0$.
So assume $U_3\neq 0$.
Note that in this case, $\det(U_1 + ZU_3)$ is not a constant.
Suppose that $g\circ U(0) \neq 0$. Then $ \det(I_{r'} - g\circ U(Z) (g\circ U(0))^*)$ is not a constant and hence it
should be of the form $p/q$ where $p$ and $q$ are non-constant polynomial without common factors
and $q = \det(U_1+ZU_3)^l$ for some $l$.
But since product of $p/q$ and polynomial cannot be a constant, \eqref{formula_2} induces a contradiction.
Hence $g\circ U(0)$ should be zero. This implies that
$0 = g(U_1^{-1} U_2) = V\circ f(0) = V(0) = V_1^{-1} V_2 $. Hence $V_2=0$ (and also $V_3=0$)
and hence $ 1= \det(V_1 + f(Z)V_3)$. Put this in \eqref{formula}.
\begin{equation}\label{equation}
\det(I_{r'}-f(Z)f(W)^*) = \det(I_{r'} - g\circ U(Z) (g\circ U(W))^*).
\end{equation}
Since right side of \eqref{equation} is singular on $\{Z\in \Omega_{r,s} : \det(U_1 + ZU_3) =0 \}$,
$U_3$ should be zero.
\end{proof}

\section{Application}
In this section, we suggest examples which are 1-parameter family of inequivalent
proper holomorphic maps between bounded symmetric domains of classical type.
We use Theorem \ref{main theorem} to prove that proper holomorphic maps $f_t:\Omega_1\rightarrow \Omega_2$
are inequivalent for each $0\leq t\leq 1$.
Every example in this section is obtained in \cite{Seo}.
As in \cite{D'Angelo_Lebl2}, we define the following equivalence relation:\\
Let $\Omega_1,\Omega_2$ be domains.
Consider a continuous map $H : \Omega_1 \times [0,1] \rightarrow \Omega_2$.
Denote $H_t(z) = H(z,t)$.
Suppose that $H_t : \Omega_1 \rightarrow \Omega_2$ is holomorphic for each $t\in [0,1]$.
Then we will say $H_t$ is a \textit{continuous family} of holomorphic maps from $\Omega_1$ to $\Omega_2$.
\begin{definition}
Let $f : \Omega_1 \rightarrow \Omega_2$ and $g:\Omega_1\rightarrow \Omega_3$ be proper holomorphic maps.
Then $f$ and $g$ are \textit{homotopic in the target domain $\Omega$} if for each $t\in [0,1]$,
there is a proper holomorphic maps $H_t : \Omega_1\rightarrow \Omega$ such that
\begin{itemize}
\item there are totally geodesic embeddings $ e_k : \Omega_k \rightarrow \Omega$ for $k=2,3$, with respect to their Bergman metrics,
\item $H_0=e_2\circ f$ and $H_1 = e_3\circ g$,
\item $H_t$ is a continuous family of holomorphic maps from $\Omega_1$ to $\Omega$.
\end{itemize}
\end{definition}

\subsection{1-parameter family of proper holomorphic maps among $\Omega_{r,s}^\Rmnum{1}$.}
\mbox{}

Consider proper holomorphic maps $f,g :\Omega_{2,2}^\Rmnum{1}\rightarrow \Omega_{3,3}^\Rmnum{1}$ which are defined by
$$f\left(\left(
         \begin{array}{cc}
           z_1 & z_2 \\
           z_3 & z_4 \\
         \end{array}
       \right)
    \right) = \left(
           \begin{array}{ccc}
             z_1^2 & z_1z_2& z_2  \\
             z_1z_3 & z_2z_3& z_4  \\
             z_3 & z_4& 0  \\
           \end{array}
         \right), \text{ for }
         \left(
         \begin{array}{cc}
           z_1 & z_2 \\
           z_3 & z_4 \\
         \end{array}
       \right) \in \Omega_{2,2}^\Rmnum{1}$$
$$g\left(\left(
         \begin{array}{cc}
           z_1 & z_2 \\
           z_3 & z_4 \\
         \end{array}
       \right)
    \right) = \left(
           \begin{array}{ccc}
             z_1^2 & \sqrt{2}z_1z_2 & z_2^2 \\
             \sqrt{2} z_1z_3 & z_1z_4+z_2z_3 & \sqrt{2}z_2z_4 \\
             z_3^2 & \sqrt{2}z_3z_4 & z_4^2 \\
           \end{array}
         \right)$$
and $f_t : \Omega_{2,2}^{\Rmnum{1}}\rightarrow \Omega_{4,4}^{\Rmnum{1}}$ be proper holomorphic maps
for $ 0\leq t\leq 1$ defined by
\begin{equation}\label{inequivalent example}
  \left(
         \begin{array}{cc}
           z_1 & z_2 \\
           z_3 & z_4 \\
         \end{array}
       \right) \mapsto \left(
         \begin{array}{cccc}
           z_1^2 & \sqrt{2-t}z_1z_2 & \sqrt{1-t} z_2^2& \sqrt{t}z_2 \\
           \sqrt{2-t}z_1z_3 & \frac{2(1-t)}{2-t}z_1z_4 + z_2z_3& 2\sqrt{\frac{1-t}{2-t}}z_2z_4& \sqrt{\frac{t}{2-t}}z_4 \\
           \sqrt{1-t}z_3^2 & 2\sqrt{\frac{1-t}{2-t}}z_3z_4& z_4^2& 0 \\
           \sqrt{t}z_3 & \sqrt{\frac{t}{2-t}}z_4&0&0\\
         \end{array}
       \right).
\end{equation}
Then it is easily observed that $f$ and $g$ are homotopic in the target domain $\Omega_{4,4}^{\Rmnum{1}}$
through $f_t$.

\begin{cor} \label{non-equiv theorem2}
 $f_t$ are inequivalent for different $t$, $0\leq t \leq 1$.
\end{cor}
\begin{proof}
Suppose that $f_t\circ A = B\circ f_s$ for some $A \in U(2,2)$ and $B \in U(4,4)$.
Without loss of generality, we may assume that $t\neq 0$.
Then by Theorem \ref{main theorem}, $f_t(UZV) = Lf_s(Z)M$ for some $U=\left(
                                                                                        \begin{array}{cc}
                                                                                          U_1 & U_2 \\
                                                                                          U_3 & U_4 \\
                                                                                        \end{array}
                                                                                      \right)\in U(2)$,
$V = \left(
       \begin{array}{cc}
         V_1 & V_2 \\
         V_3 & V_4 \\
       \end{array}
     \right)
\in U(2)$
and $L=(L_{ij}),\,M=(M_{ij})\in U(4)$.
Denote $f_t = \sum f_{t,j}$ where $f_{t,j}$ is a homogeneous polynomial of degree $j$.
Then
$f_{t,j}(UZV) = Lf_{s,j}(Z)M$ for each $j$. Consider linear part
$$f_{t,1} =  \left(
         \begin{array}{cccc}
           0 & 0 & 0& \sqrt{t}z_2 \\
           0 & 0& 0& \sqrt{\frac{t}{2-t}}z_4 \\
          0& 0& 0& 0 \\
           \sqrt{t}z_3 & \sqrt{\frac{t}{2-t}}z_4&0&0\\
         \end{array}
       \right).$$
At first, at $ Z_1 = \left(
           \begin{array}{cc}
             1 & 0 \\
             0 & 0 \\
           \end{array}
         \right)$,
 $$f_{t,1}(UZ_1V) = f_{t,1}\left(\begin{array}{cc}
                              U_1V_1 & U_1V_2 \\
                              U_3V_1 & U_3V_2
                            \end{array}\right) =  \left(
         \begin{array}{cccc}
           0 & 0 & 0& \sqrt{t}U_1V_2 \\
           0 & 0& 0& \sqrt{\frac{t}{2-t}} U_3V_2 \\
          0& 0& 0& 0 \\
            \sqrt{t}U_3V_1 & \sqrt{\frac{t}{2-t}}U_3V_2&0&0\\
         \end{array}
       \right)$$
and $Lf_{s,1}(Z_1)M =0$. This implies that $U_1V_2=U_3V_2=U_3V_1=U_3V_2=0$. Suppose that $U_3 \neq 0$.
Then $V_1 = V_2 = 0$. This is a contradiction to $V\in U(2)$. Hence $U_3=0$. This implies that $U_2=0$ and $U_1\neq 0$, hence $V_2=V_3=0$.

Second, at $Z_2 = \left(
                          \begin{array}{cc}
                            0 & 1 \\
                            0 & 0 \\
                          \end{array}
                        \right)$,
$$f_{t,1}(UZ_2V) = f_{t,1}\left(\begin{array}{cc}
                              0 & U_1V_4 \\
                              0 & 0
                            \end{array}\right) =  \left(
         \begin{array}{cccc}
           0 & 0 & 0& \sqrt{t}U_1V_4 \\
           0 & 0& 0& 0 \\
           0 & 0& 0 & 0 \\
           0 & 0&0&0\\
         \end{array}
       \right)$$ and
$$Lf_{s,1}(Z_2)M  = \sqrt{s}\left(
                                           \begin{array}{c}
                                             L_{11}(M_{41},\, M_{4 2},\, M_{43},\, M_{44}) \\
                                             L_{21}(M_{41},\, M_{4 2},\, M_{43},\, M_{44}) \\
                                             L_{31}(M_{41},\, M_{4 2},\, M_{43},\, M_{44}) \\
                                             L_{41}(M_{41},\, M_{4 2},\, M_{43},\, M_{44}) \\
                                           \end{array}
                                         \right).
$$
This implies that $L_{21}=L_{31}=L_{41}=0$, $L_{11} \neq 0$, $M_{41}=M_{42}=M_{43}=0$ and $\sqrt{t}U_1V_4 = \sqrt{s} L_{11}M_{44}$.
Since $L$ is unitary, $L_{12}=L_{13}=L_{14}=0$. Hence by taking square norm on $\sqrt{t}U_1V_4 = \sqrt{s} L_{11}M_{44}$, we obtain $t=s$.
\end{proof}

\begin{remark}
In \cite{Seo}, there are generalizations of $f_0$, $f_1$ to $\tilde{f}_0$, $\tilde{f}_1$
as proper holomorphic maps
$\tilde{f}_0 : \Omega_{r,s}^{\Rmnum{1}}\rightarrow \Omega_{2r-1,2s-1}^{\Rmnum{1}}$ and
$\tilde{f}_1 : \Omega_{r,s}^{\Rmnum{1}} \rightarrow \Omega_{\frac{1}{2}r(r+1), \frac{1}{2}s(s+1)}^{\Rmnum{1}}$.
Similar to \eqref{inequivalent example}, we can make homotopy $\tilde{f}_t$ of $\tilde{f}_0$ and $\tilde{f}_1$ which are inequivalent for all $t$, $0\leq t\leq 1$.
That is, $\tilde f_0$ and $\tilde f_1$ are homotopic in the target domain $\Omega^\Rmnum{1}_{r',s'}$ with $r'=\frac{1}{2}(r^2+r+2s-2)$ and $s'=\frac{1}{2}(s^2+s+2r-2)$.
\end{remark}

There are only two inequivalent proper holomorphic maps from $\BB^n$ to $\BB^{2n-1}$ which are the standard embedding and the Whitney map (see \cite{Huang_Ji}).
As the case of the unit balls, it is expected that there are only two inequivalent proper holomorphic maps from $\Omega_{r,s}^\Rmnum{1}$ to $\Omega_{2r-1, 2s-1}^\Rmnum{1}$ that are the standard embedding and the generalized Whitney map \eqref{generalized Whitney map}. In contrast to $\Omega_{2r-1, 2s-1}^\Rmnum{1}$,
if the target domain changes to $\Omega_{2r-1,2s}$, there are infinitely many proper holomorphic maps;
 Let $g_t :\Omega_{2,2}^{\Rmnum{1}} \rightarrow \Omega_{3,4}^{\Rmnum{1}}$, $0\leq t\leq 1$ be proper holomorphic maps defined by
\begin{equation}\label{inequivalent example2}
 \left(
         \begin{array}{cc}
           z_1 & z_2 \\
           z_3 & z_4 \\
         \end{array}
       \right) \mapsto
\left(
         \begin{array}{cccc}
           \sqrt{t}z_1^2 & \sqrt{t}z_1z_2 & \sqrt{1-t} z_1& z_2 \\
           \sqrt{t}z_1z_3 & \sqrt{t}z_2z_3& \sqrt{1-t}z_3& z_4 \\
           z_3& z_4& 0 & 0 \\
         \end{array}
       \right).
\end{equation}
$g_t$ are generalized to proper holomorphic map
$G_t :\Omega_{r,s}^{\Rmnum{1}}\rightarrow \Omega_{2r-1,2s}^{\Rmnum{1}}$, $0\leq t\leq 1$, given by for $Z = (z_{ij})_{1\leq i\leq r,\, 1\leq j\leq s}$.
\begin{equation} \label{generalized whitney map}
 Z
\mapsto\left(
  \begin{array}{cccccccc}
    \sqrt{t}z_{11}^2  & \sqrt{t}z_{11}z_{12}&\dots& \sqrt{t}z_{11}z_{1s}&\sqrt{1-t}z_{11}& z_{12}& \dots & z_{1s}   \\
    \sqrt{t}z_{11} z_{21} & \sqrt{t}z_{21}z_{12}& \dots& \sqrt{t}z_{21}z_{1s} &\sqrt{1-t}z_{21}&z_{22}& \dots& z_{2s}  \\
    \vdots & \vdots &\ddots & \vdots&\vdots& \ddots &\vdots \\
    \sqrt{t}z_{11}z_{r1} & \sqrt{t}z_{r1}z_{12} &\dots & \sqrt{t}z_{r1}z_{1s}&\sqrt{1-t}z_{r1}& z_{r2}& \dots& z_{rs} \\
    z_{21} & z_{22} &\dots & z_{2s}& 0&0& \dots & 0 \\
     \vdots& \vdots & \ddots& \vdots &\vdots& \vdots& \ddots & \vdots\\
     z_{r1}& z_{r2}&\dots &z_{rs}&0&0& \dots& 0\\
  \end{array}
\right)
\end{equation}

\begin{proof}[Proof of Theorem~\ref{non-equiv theorem}]
The same method of the proof in Corollary \ref{non-equiv theorem2} can be applied. We omit the proof.
\end{proof}

\subsection{1-parameter family of proper holomorphic maps among $\Omega_n^{\Rmnum{3}}$.}
\mbox{}

Consider the proper holomorphic maps $h_0 : \Omega_{2}^{\Rmnum{3}}\rightarrow \Omega_3^{\Rmnum{3}}$
and $h_1 : \Omega_{2}^{\Rmnum{3}}\rightarrow \Omega_3^{\Rmnum{3}}$ defined by
$$
h_0\left(\left(
         \begin{array}{cc}
           z_1 & z_2 \\
           z_2 & z_3 \\
         \end{array}
       \right)
    \right) = \left(
           \begin{array}{ccc}
             z_1^2 & \sqrt{2}z_1z_2 & z_2^2 \\
             \sqrt{2} z_1z_2 & z_1z_2+z_2^2 & \sqrt{2}z_2z_3 \\
             z_2^2 & \sqrt{2}z_2z_3 & z_3^2 \\
           \end{array}
         \right)$$
$$ h_1\left(\left(
         \begin{array}{cc}
           z_1 & z_2 \\
           z_2 & z_3 \\
         \end{array}
       \right)
    \right) = \left(
           \begin{array}{ccc}
             z_1^2 & z_1z_2& z_2  \\
             z_1z_2 & z_2^2& z_3  \\
             z_2 & z_3& 0  \\
           \end{array}
         \right),\,\,\,
\text{for} \left(
         \begin{array}{cc}
           z_1 & z_2 \\
           z_2 & z_3 \\
         \end{array}
       \right) \in \Omega_{2,2}$$
which are the restriction of $f_0$ and $f_1$ to $\Omega_2^{\Rmnum{3}}$.
These are homotopically equivalent in the target domain $\Omega_4^{\Rmnum{3}}$ by $h_t : \Omega^{\Rmnum{3}}_2\rightarrow \Omega^{\Rmnum{3}}_4$ which are proper holomorphic maps
given by $h_t(Z) = f_t|_{\Omega^{\Rmnum{3}}_2}(Z)$. Explicitly,
\begin{equation}
  h_t : \left(
         \begin{array}{cc}
           z_1 & z_2 \\
           z_2 & z_3 \\
         \end{array}
       \right) \mapsto \left(
         \begin{array}{cccc}
           z_1^2 & \sqrt{2-t}z_1z_2 & \sqrt{1-t} z_2^2& \sqrt{t}z_2 \\
           \sqrt{2-t}z_1z_2 & \frac{2(1-t)}{2-t}z_1z_3 + z_2^2& 2\sqrt{\frac{1-t}{2-t}}z_2z_3& \sqrt{\frac{t}{2-t}}z_3 \\
           \sqrt{1-t}z_2^2 & 2\sqrt{\frac{1-t}{2-t}}z_2z_3& z_3^2& 0 \\
           \sqrt{t}z_2 & \sqrt{\frac{t}{2-t}}z_3&0&0\\
         \end{array}
       \right).
\end{equation}
We can apply similar method to $h_t$ and a higher version of $h_t$, we obtain the following:
\begin{cor}
There are uncountably many inequivalent proper holomorphic maps from $\Omega_n^{\Rmnum{3}}$ to $\Omega_{n'}^{\Rmnum{3}}$ with $n' = \frac{1}{2}(n^2+3n-2)$.
\end{cor}

\end{document}